\documentclass[11pt]{article}
\usepackage{amsmath}
\usepackage{amsfonts}
\usepackage{amssymb}
\usepackage[mathscr]{eucal}
              \newif{\ifcomentarios}\comentariosfalse\newtheorem{theorem}{Theorem}
\newtheorem{teo}[theorem]{Theorem}

\newtheorem{lemma}[theorem]{Lemma}
\newtheorem{remark}[theorem]{Remark}

\newtheorem{proposition}[theorem]{Proposition}
\newenvironment{proof}[1][Proof]{\noindent\textit{#1.}}{\hfill $\Box$}

\newcommand{\ud}{\mathrm{d}}
\newcommand{\N}{\mathbb{N}}
\newcommand{\Z}{\mathbb{Z}}
\newcommand{\R}{\mathbb{R}}
\newcommand{\C}{\mathbb{C}}
\newcommand{\e}{\epsilon}
\newcommand{\dis}{\displaystyle}
\newcommand{\pt}{\partial}
\newcommand{\ii}{\mathrel{\mathop:}}
\newcommand{\LL}{\mathrm L}
\newcommand{\CC}{\mathrm C}

\begin{document}\title{On the spectrum of a Robin Laplacian in a planar waveguide}
\author{Alex F. Rossini}
\date{}\maketitle

\abstract{We consider the Laplace operator in a planar waveguide, i.e., an infinite two-dimensional straight strip of constant width,  with particular types of Robin boundary conditions. We study the essential spectrum of the corresponding Laplacian when the boundary coupling function has a limit at infinity. Furthermore, we derive sufficient conditions for the existence of discrete spectrum.

\


\

\noindent {\bf Running-head title:} Robin planar waveguide.

\


\section{Introduction}
There are different ways of confining a quantum particle in long and thin structures, the so-called quantum waveguides in suitable subsets $\Omega_\e$ of the space~$\mathbb R^3$ or the plane~$\mathbb R^2$  \cite{CA,FS,FSR,GJ,DK,KK}.  A usual possibility, in two dimensions, is to model the waveguide by a curved strip of constant width which is squeezed between two curves; in this region one considers the Laplacian subject to Dirichlet~\cite{CA,DE,FS}, Neumann~\cite{ES,DK} or Robin boundary conditions~\cite{CAL,FK,MJ}.  
\par Our main interest in this paper is to describe the precise location of the essential spectrum of a  Robin Laplacian $-\Delta_{\alpha}^{\Omega_\e}$, and  study the existence of eigenvalues below the essential spectrum, in a straight quantum waveguide $\Omega_\e$; see \cite{CDFK,FK,MJ,DZ} for related  references. 

\par The description of the model here studied is as follows. Given a positive number~$\e,$ consider the infinite straight strip $\Omega_\e= \R\times I,$ where $I=(0,\e)$ is a bounded interval. The operator $-\Delta_{\alpha}^{\Omega_\e}$  acts as the Laplacian in the Hilbert space  $\LL^{2}(\Omega_\e)$ with certain Robin conditions at the boundary $\pt\Omega_{\e}.$ More specifically, given a bounded real-valued function~$\alpha(x)$ on~$\R,$ the classical version of such conditions are  
\begin{equation}\label{R1}
\left \{
\begin{array}{cc}
-\dis\frac{\pt \psi}{\pt y}(x,0)- \alpha(x) \psi(x,0)=0\\ 
\,\,\,\,\dis\frac{\pt \psi}{\pt y}(x,\epsilon)+  \alpha(x) \psi(x,\epsilon)=0
\end{array}
\right.,
\end{equation}
for each $x\in\R$ and each $\psi\in\mathrm{dom}\,(-\Delta_{\alpha}^{\Omega_\e})$. A related type of boundary conditions has been considered in~\cite{MJ}; there the author has investigated spectral properties of the Laplacian by imposing (usual) Robin conditions, i.e.,  without changing the sign of the Robin parameter~$\alpha$ as in~\eqref{R1}, whose classical version are
\begin{equation}\label{R2}
\left \{
\begin{array}{cc}
-\dis\frac{\pt \psi}{\pt y}(x,0)+\alpha(x) \psi(x,0)=0\\ 
\,\,\,\,\dis\frac{\pt \psi}{\pt y}(x,\epsilon)+  \alpha(x) \psi(x,\epsilon)=0
\end{array}
\right.,
\end{equation}
where $\alpha(x)$ is positive for all $x\in\R.$
Considering this case and under the hypothesis that~$\alpha$ tends to a constant at infinity, the essential spectrum of the Laplacian was determined and  a sufficient condition for the existence of discrete spectrum was given. The strategy  in~\cite{MJ} to prove the existence of  at least one isolated eigenvalue, below the threshold of the essential spectrum, was  a variational one based on~\cite{DE}, and the method of Neumann Bracketing was employed to find the location of the essential spectrum. 

\par It is a question here whether there are any similar results when one chooses our boundary conditions~\eqref{R1}. This change of sign of the Robin parameter leads to nonpositive quadratic forms and, in this context, we were able to get similar results as in~\cite{MJ}. Note that here the location of the essential spectrum was obtained by means the an idea in~\cite{BK}.

\par Let $-\Delta_{\alpha}^{\Omega_\e}$  denote the Laplacian  with $\mathrm{dom}\,(-\Delta_{\alpha}^{\Omega_{\epsilon}})=\{\psi \in H^{2}(\Omega_\epsilon)\,;\psi\,\,\mbox{satisfies}\,\,(\ref{R1})\}.$ Note that $-\alpha_{0}^{2}$ is the first eigenvalue of the Laplacian $-\Delta_{\alpha_{0}}^{I}$ in $\LL^{2}(I)$ with  $\psi(x)\in\mathrm{dom}(-\Delta_{\alpha_{0}}^{I})$ if $\phi\in  H^{2}(I)$ and satisfying 
\begin{equation}\label{R1R}
\left \{
\begin{array}{cc}
-\psi'(0)-\alpha_{0}\psi(0)=0\\ 
\,\,\,\,\psi'(\e)+\alpha_{0} \psi(\e)=0
\end{array}
\right..
\end{equation}
\par We are going to show that, under some conditions as $\lim_{|x|\to+\infty}(\alpha-\alpha_0)=0$,
$$ \sigma_{\mathrm{ess}}(-\Delta_{\alpha}^{\Omega_\e})=\sigma_{\mathrm{ess}}(-\Delta_{\alpha_0}^{\Omega_\e})=[-\alpha_{0}^{2},+\infty) $$ and
$$\sigma(-\Delta_{\alpha}^{\Omega_{\e}})\cap (-\infty, -\alpha_{0}^{2})\neq\emptyset.$$ 
\par The  operators are introduced as the unique self-adjoint operators associated with appropriate quadratic forms and the boundary conditions should be understood in the sense of traces (see more details in Sections~\ref{LR} and~\ref{LRF}). 
\par The paper is organized as follow. In Section~\ref{LR} we introduce a Robin Laplacian in a bounded interval (transversal section), show that its essential spectrum is empty and explicitly compute  its eigenvalues. The change of sign of the Robin parameter~$\alpha$ leads to a negative first eigenvalue  (and equal to~$-\alpha^2 $), whereas all the others are positive. In Section~\ref{LRF} we pass to the corresponding study in an infinite straight and narrow strip. We show, via quadratic forms, that the operator $-\Delta_{\alpha}^{\Omega_\e}$ is self-adjoint (Theorem~\ref{alpha}). Finally, in Section~\ref{discrete},  we find the essential spectrum of such Robin Laplacian operator, and gives  sufficient conditions for the existence of discrete spectrum.

\section{Transversal Robin Laplacian}\label{LR}

Initially some results will be presented to our Robin Laplacian in the interval (transversal section) $I=(0,\e);$ they will be important ahead. Here, we find that the Laplacian operator $-\Delta_{\alpha}^{I}$ (classic) in $\LL^{2}(I)$ is self-adjoint by using the theory of quadratic forms.
\par Consider the operator 
$$-\Delta_{\alpha}^{I}:\mathrm{dom}\,(-\Delta_{\alpha}^{I})\rightarrow \LL^{2}(I),$$
with $\mathrm{dom}\,(-\Delta_{\alpha}^{I})=\{\psi \in H^{2}(I)\,;\psi\,\,\,\, \mbox{satisfies}\,\,\, (\ref{RII}) \},$
\begin{equation}\left \{
\begin{array}{cc}\label{RII}
-\psi'(0)-\alpha \psi(0)=0\\
\,\,\,\psi'(\e)+\alpha \psi(\e)=0 
\end{array}
\right..
\end{equation}
Let $b_{\alpha}\geq -|\alpha|^2$ be the closed  sesquilinear form $\mathcal{H},$ with domain $\mathrm{dom}\,\,b_{\alpha}=H^{1}(I)\sqsubseteq \LL^{2}(I),$ 
\begin{eqnarray*}
b_{\alpha}(\phi,\psi)=\int_{0}^{\e}\overline{ \phi'(y)}\,\psi'(y)\,\ud y+\alpha \big(\overline{\phi(\e)}\psi(\e)-\overline{\phi(0)}\psi(0)\big).
\end{eqnarray*}

\begin{teo}\label{laplaciano1}
Let   $\alpha\in \R-\{0\}.$ Then, the (negative) Laplacian operator $-\Delta_{\alpha}^{I}$  is the unique  self-adjoint operator associated  with the sesquilinear form $b_{\alpha},$ that is, 
$$b_{\alpha}(\phi,\psi)=(\phi,-\Delta_{\alpha}^{I}\psi),$$
for each $\phi\in \mathrm{dom}\,\,b_{\alpha}$, and $\psi\in\mathrm{dom}\,(-\Delta_{\alpha}^{I}).$ 
\end{teo}
\begin{proof}\,\, We will first prove that $b_{\alpha}\,$ is closed and lower bounded. Denote by $\|\,\,\|$ the norm in $\LL^{2}(I);$ we have  that 
$$b_{\alpha}(\phi)\geq -|\alpha|^{2}\|\phi\|^2\,,\,\,\ \forall\,\,\, \phi\in \mathrm{dom}\,b_{\alpha}.$$  Indeed, first note that for all $a>0$ we have
$2\|\phi'\|\,\|\phi\|\leq a^2\|\phi\|^2+\frac{1}{a^2}\|\phi'\|^2.$ If $\phi\in \mathrm{dom}\,b_{\alpha},$ then
\begin{align*}
\Big(|\phi(1)|^2-|\phi(0)|^2\Big)&=\int_{0}^{1}\frac{\ud}{\ud x}|\phi|^2\,\ud x=\int_{0}^{1}(\overline{\phi'}\phi+\phi'\overline{\phi})\,\ud x\\
{}                   &\leq 2\|\phi'\|\|\phi\|\leq a^2\|\phi\|^2+\frac{1}{a^2}\|\phi'\|^2.
\end{align*}
Consequently,
\begin{align*}
b_{\alpha}(\phi)&\geq \left(1-\frac{|\alpha|}{a^2}\right)\|\phi'\|^2-|\alpha|a^2\|\phi\|^2.
\end{align*}
Choosing $a=\sqrt{|\alpha|}\,,$ we obtain $b_{\alpha}(\phi)\geq-|\alpha|^2\|\phi\|^2.$
Note that by choosing $a=\sqrt{2|\alpha|},$ we will have
\begin{eqnarray}\label{close}
b_{\alpha}(\phi)\geq \frac{1}{2}\|\phi'\|^2-2|\alpha|^{2}\|\phi\|^2.
\end{eqnarray}
By inequality (\ref{close}) and since $H^{1}(I)\subset\CC[0,\e],$ with continuous injection, it follows that $b_{\alpha}$ is closed.
\par Denote by $T_{b_{\alpha}}$ the unique  self-adjoint operator associated with the form $b_{\alpha},$ then $T_{b_{\alpha}}=-\Delta_{\alpha}^{I}$ with $\mathrm{dom}\,T_{b_{\alpha}}=\mathrm{dom}\,(-\Delta_{\alpha}^{I}).$ Indeed, if $\psi\in \mathrm{dom}\,(-\Delta_{\alpha}^{I})$ then,  by integration by parts, we have the equality 
$b_{\alpha}(\phi,\psi)=(\phi,-\psi'')_{\LL^{2}(I)};$ whence we obtain
$T_{b_{\alpha}}|_{\mathrm{dom}\,(-\Delta_{\alpha}^{I})}=-\Delta_{\alpha}^{I},$ because $\mathrm{dom}\,(-\Delta_{\alpha}^{I})\subseteq \mathrm{dom}\,T_{b_{\alpha}},$ and so $T_{b_{\alpha}}=-\Delta_{\alpha}^{I}$.  
\par Let $\psi\in\mathrm{dom}\,T_{b_{\alpha}}$ and $\eta\ii=T_{b_{\alpha}}\psi\in \LL^{2}(I)$ then $(\phi,\eta)_{\LL^{2}(I)}= b_{\alpha}(\phi,\psi),$ i.e,
\begin{equation}\label{OpA}
\int_{0}^{\e}\bar{\phi}\eta\,\ud x=\int_{0}^{\e}\bar{\phi'}\psi'\,\ud x+
\alpha\Big(\bar{\phi}(\e)\psi(\e)-\bar{\phi}(0)\psi(0)\Big).
\end{equation}
Following the ideas in~\cite{Ka}, Example VI. 2.16, let $z\in H^{1}(I)$ a primitive of $\eta,$ i.e., $z'=\eta$ , then
\begin{equation}\label{OpA2}
\int_{0}^{\e}\bar{\phi}\eta\,\ud x=\int_{0}^{\e}\bar{\phi} z'\,\ud x=-\int_{0}^{\e}\bar{\phi'}z\,\ud x+\Big(\bar{\phi}(\e)z(\e)-\bar{\phi}(0)z(0)\Big).
\end{equation}
Comparing (\ref{OpA}) and (\ref{OpA2}) the equality
\begin{equation}\label{OpA3}
\int_{0}^{\e}\bar{\phi'}(\psi'+z)\,\ud x+ \bar{\phi}(\e)[\alpha\psi(\e)-z(\e)]+ \bar{\phi}(0)[-\alpha\psi(0)+z(0)]=0,\,\,\forall\,\,\phi\in H^{1}(I),
\end{equation}
holds true.
In particular, if $\phi\in \CC_{0}^{\infty}(0,\e)$ then $\int_{0}^{\e}\bar{\phi '}(\psi'+z)\,\ud x=0,$ so, 
$\psi'+z=c\quad\mbox{a.e}\quad x\in I.$
Therefore, $\psi\in H^{2}(I),$ because $\psi'=c-z\in H^{1}(I).$ Moreover,
$\psi''=-z'=-\eta$ and
\begin{equation}\left \{
\begin{array}{cc}\label{Contour}
\psi'(0)+z(0)= c\\
\psi'(\e)+z(\e)=c
\end{array}
\right..
\end{equation} 
\par Thus, only remains to verify the Robin condition for $\psi,$  and will follow the inclusion $\mathrm{dom}\,T_{b_\alpha}\subseteq\mathrm{dom}\,(-\Delta_{\alpha}^{I}).$
Finally, if we replace $(\psi'+z)=c$ in (\ref{OpA3}) and use integration by parts we deduce
\begin{equation}\label{OpA4}
\bar{\phi}(\e)[c+\alpha\psi(\e)-z(\e)]+\bar{\phi}(0)[-c-\alpha\psi(0)+z(0)]=0,\,\,\forall\,\,\phi\in H^{1}(I).
\end{equation}
Since $\phi$ is arbitrary and in view of~(\ref{Contour}) the desired conclusion follows.
\end{proof}

\begin{remark}
Note that  Theorem~\ref{laplaciano1}  holds for $\alpha=0,$ i.e., the  Laplacian $-\Delta_{N}^{I}$ with Neumann condition, which is the operator associated with $$b_{N}(\phi)=\int_{0}^{\e}|\phi'|^2\,\ud y,$$ $\mathrm{dom}\,b_{N}=H^{1}(I)$ and  
$\mathrm{dom}\,(-\Delta_{N}^{I})=\{\psi\in H^{2}(I), \psi'(0)=\psi'(\e)=0\}.$ In addition, we know that the eigenvalues of $-\Delta_{N}^{I}$ rearranged in an ascending order are given by $$\lambda_{0}^{N}=0,\,\,\,\lambda_{n}^{N}=\frac{n^2\pi^2}{\e^2}, \,\,\mbox{with}\,\,n\in\N.$$ The corresponding normalized eigenfunctions are:
\begin{equation}
\psi_{n}^{N}(y)=\left \{
\begin{array}{lc}
\sqrt{\frac{1}{\e}},& \mbox{if}\quad n=0,  \\
\sqrt{\frac{2}{\e}}\cos\left(\frac{n\pi y}{\e}\right), & \mbox{if}\quad n\geq1
\end{array}
\right..
\end{equation}
The sequence $\{\psi_{n}^{N}\}_{n=1}^{\infty}$ is a orthonormal basis of $\LL^{2}(I)$ and since $\lambda_{n}^{N}\rightarrow \infty$ when
$n\rightarrow \infty,$ we have that the essential spectrum $\sigma_{\mathrm{ess}}(-\Delta_{N}^{I})$ of operator $-\Delta_{N}^{I}$  is empty. Therefore,  the spectrum $\sigma(-\Delta_{N}^{I})$ of operator $-\Delta_{N}^{I}$ is equal to discrete spectrum $\sigma_{\mathrm{disc}}(-\Delta_{N}^{I}),$ that is,
$$\sigma(-\Delta_{N}^{I})=\sigma_{\mathrm{disc}}(-\Delta_{N}^{I})=\{\lambda_{n}^{N}\}_{n=0}^{\infty}.$$ For future reference, let
$\psi_{n}^{D}(x)=\sqrt{\frac{2}{\e}}\sin\big(({n\pi y})/\e\big),n\geq1,$ the normalized eigenfunctions associated with the Dirichlet Laplacian $-\Delta_{D}^{I},$ i.e., the operator associated with the quadratic form $$b_{D}(\phi)=\int_{0}^{\e}|\phi'|^2\,\ud x,$$ $\mathrm{dom}\,b_{D}=H_{0}^{1}(I)$ and
$\mathrm{dom}\,(-\Delta_{D}^{I})=\{\psi\in W^{2,2}(I), \psi(0)=\psi(\e)=0\}.$
 See~\cite{KK} for more details.
\end{remark}

\subsection{Absence of Essential spectrum}
Here we determine the essential spectrum of our transveral Robin Laplacian operator $\Delta_{\alpha}^{I}$. Inspired by~\cite{KBZ}, see section $(6.2),$ we will conclude  that $\sigma_{\mathrm{ess}}(-\Delta_{\alpha}^{I})=\emptyset.$ To prove this fact, we need the following result (Lemma~\ref{Kato} ahead) from~\cite{Ka} Theorem~V.2.20.


\begin{lemma}\label{Kato}
Let $\{\psi_{n}^{N}\}_{n=0}^{\infty}$ be a complete orthonormal family in $\mathcal{H}$ (Hilbert space) and let $\{\psi_{n}\}_{n=0}^{\infty}$ be a sequence such that $\sum_{n=0}^{\infty} \|\psi_{n}-\psi_{n}^{N}\|_{2}^{2} < \infty.$  Then, $\psi_{n}$ is a basis of $\mathcal{H}$, and so, if $0=\sum_{n=0}^{\infty}c_{n}^{\psi}\psi_n$, then all $c_{n}^{\psi}=0.$ 
\end{lemma}

\par We will see in the next lemma, that the elements
$$\psi_{0}(y)=ce^{-\alpha y},\quad \psi_{n}(y)=\frac{n\pi}{(n^2\pi^2+\alpha^2\e^2)^{1/2}}\left(\psi_{n}^{N}-\frac{\alpha\e}{n\pi}\psi_{n}^{D}\right),n\geq 1$$ 
consist of an orthonormal basis of $\LL^{2}(I).$ Moreover, in Subsection~\ref{pontualS}, theses elements are shwon to be eigenvectors of $-\Delta_{\alpha}^{I},$ associated with the eigenvalues, respectively, $$\lambda_{0}=-\alpha^2\quad\mbox{and}\quad\lambda_{n}=(n^2\pi^2)/\e^2, n\geq1.$$ 

\begin{lemma}
The sequence $\{\psi_n\}_{n=0}^{\infty}$ is a complete orthonormal family in $\LL^{2}(I).$ 
\end{lemma}
\begin{proof}\,\,
Indeed, given $\phi\in \LL^{2}(I)$ and $\psi=\sum_{n=0}^{\infty}c_{n}^{\psi}\psi_n$  it follows that
\begin{equation}\label{3.1}
(\phi,\psi)_{\LL^{2}(I)}=\lim_{m\rightarrow \infty}(\phi,\sum_{n=0}^{m} c_{n}^{\psi}\psi_n)_{\LL^{2}(I)}.
\end{equation}
For $\psi=0$ and $\phi=\psi_{k},$ it follows by (\ref{3.1}) and orthogonality that $c_{k}^{\psi}=0,$ for each $\, k=0,1,...$ 
On the other hand, we have
$$\sum_{n=0}^{\infty} \|\psi_{n}-\psi_{n}^{N}\|_{\LL^2(I)}^{2} < \infty.$$ Indeed, for $n\geq 1,$ we have
$$\|\psi_{n}-\psi_{n}^{N}\|_{\LL^2(I)}^{2}=2-\frac{2n\pi}{\sqrt{n^2\pi^2+\alpha^2\e^2}}\rightarrow 0,\qquad n\rightarrow \infty.$$
Since  the function $f(x)=2-\dis\frac{2x\pi}{\sqrt{x^2\pi^2+\alpha^2\e^2}}$ is decreasing and $\int_{1}^{\infty}f(x)\,\ud x< \infty,$ the integral test for convergence  give us that $\Sigma_{n=1}^{\infty}f(n)<+\infty,$ consequently,
$\sum_{n=0}^{\infty} \|\psi_{n}-\psi_{n}^{N}\|_{\LL^2(I)}^{2}$ is convergent. Therefore, by Lemma~\ref{Kato},  the desired result follows.
\end{proof}

\begin{teo}\label{autofun}
Let us suppose  $\alpha\in \R-\{0\},$ then the transversal Robin Laplacian $-\Delta_{\alpha}^{I}$ has purely discrete spectrum, i.e., the essential spectrum $\sigma_{\mathrm{ess}}(-\Delta_{\alpha}^{I})$ of operator $-\Delta_{\alpha}^{I}$  is empty.
\end{teo}
\begin{proof}\,\,  
 By Theorem~11.3.13 in~\cite{CEsar}, it follows that $\sigma_{\mathrm{ess}}(-\Delta_{\alpha}^{I})=\emptyset,$ since the sequence $\{\psi_{n}\}_{n=0}^{\infty}$ is a orthonormal basis of $\LL^{2}(I)$ and  $\lambda_{n}\rightarrow \infty$ when $n\to\infty.$ Therefore, 
$$\sigma(-\Delta_{\alpha}^{I})=\sigma_{\mathrm{disc}}(-\Delta_{\alpha}^{I})=\{\lambda_{n}\}_{n=0}^{\infty}.$$
\end{proof} 

\subsection{Point spectrum of the transversal Laplacian}\label{pontualS}
To obtain the point spectrum of $-\Delta_{\alpha}^{I},$ let us to determine $\lambda\in\R$ for which there exists $0\neq\psi\in H^{2}(I),$ normalized
in $\LL^{2}(I),$ satisfying
\begin{equation}\label{17J}
-\psi''=\lambda \psi
\end{equation}
and the boundary conditions ($\alpha\neq 0$),
\begin{equation}\label{18J}
\left \{
\begin{array}{cc}
-\psi'(0)-  \alpha \psi(0)=0\\
\,\,\,\,\psi'(\e)+ \alpha \psi(\e)=0 
\end{array}
\right..
\end{equation}
If $\lambda>0$ we know that the general solution (classic) of (\ref{17J}) is given by
\begin{equation}\label{19J}
\psi(x)=A\sin(\sqrt{\lambda}x)+B\cos(\sqrt{\lambda}x),
\end{equation}
with $A,B\in \C$ determined by the Robin condition and the normalization condition. Thus, by imposing the Robin conditions  on the general solution we obtain the following system
\begin{equation}\label{SI}
\left[\begin{array}{cc}
\sqrt{\lambda}  & \alpha  \\
\sqrt{\lambda}\cos(\sqrt{\lambda}\e)+\alpha\sin(\sqrt{\lambda}\e)  & \alpha\cos(\sqrt{\lambda}\e)-\lambda\sin(\sqrt{\lambda}\e) \\
\end{array} \right]
\left[\begin{array}{c}
A  \\
B  \\ 
\end{array} \right]=
\begin{array}{c}
0\,.
\end{array}
\end{equation}
Since we are  interested  in nonzero solution, we must impose that the determinant of the  above matrix is zero. This requirement enables us to obtain $\lambda$ explicitly: $$(\alpha^2 + \lambda)\sin(\sqrt{\lambda}\e)=0,$$ i.e., $$\lambda=-\alpha^{2}\quad\mbox{or}\quad\lambda=\frac{n^{2}\pi^2}{\e^2},\,\,\, n\in \Z.$$
Since the system (\ref{SI}) is equivalent to the equation $\sqrt{\lambda}A+\alpha B=0,$ we can express $A$ in function of $B,$ i.e,
$A=-\frac{\alpha}{\sqrt{\lambda}}B.$ Therefore, the corresponding eigenfunction  to  $\lambda=\frac{n^{2}\pi^2}{\e^2}, n\geq 1$, is given by
$$\psi_{n}(x)=B\left(-\frac{\alpha\e}{n\pi}\sin\big(({n\pi x})/\e\big)+\cos\big(({n\pi x})/\e\big)\right),$$ with $1=|B|^{2}\e\left[\frac{1}{2}\left(1+\frac{\alpha^2\e^2}{n^2\pi^2}\right)\right].$
\par Next, for $n\geq 1$ and choosing $B>0$ we have
$$\psi_{n}(x)=\frac{n\pi}{(n^2\pi^2+\alpha^2\e^2)^{1/2}}\left(\sqrt{\frac{2}{\e}}\cos\big(({n\pi x})/\e\big)-\frac{\alpha\e}{n\pi}\sqrt{\frac{2}{\e}}\sin\big(({n\pi x})/\e\big)\right),$$ one still has the following,
$$\psi_{n}(x)=\frac{n\pi}{(n^2\pi^2+\alpha^2\e^2)^{1/2}}\left(\psi_{n}^{N}-\frac{\alpha\e}{n\pi}\psi_{n}^{D}\right),$$
with $\psi_{n}^{D}\ii=\sqrt{\frac{2}{\e}}\sin\big(({n\pi x})/\e\big),\,\,n\geq 1.$ Recall that $\psi_{n}^{D}$ are the eigenfunctions of $-\Delta_{D}^{I}$, the Laplacian operator with Dirichlet boundary condition at the interval~$I.$

\par Suppose that $\lambda<0,$ then the general solution is of the form $$\psi(x)=Ae^{\sqrt{\mu}x}+Be^{-\sqrt{\mu}x},$$ with $\mu=-\lambda.$ Imposing the boundary conditions we obtain the system
\begin{equation}\label{SII}
\left[\begin{array}{cc}
\sqrt{\mu}+\alpha  & \alpha-\sqrt{\mu} \\
\sqrt{\mu}e^{\sqrt{\mu}\e}+\alpha e^{\sqrt{\mu}\e}  & \alpha e^{-\sqrt{\mu}\e}-\sqrt{\mu} e^{-\sqrt{\mu}\e} \\
\end{array} \right]
\left[\begin{array}{c}
A  \\
B  \\ 
\end{array} \right]=
\begin{array}{c}
0.
\end{array}
\end{equation}
Again, assuming that the determinant is null, follows the equality: $$(\alpha^{2}-\mu)(e^{-\sqrt{\mu}\e}-e^{\sqrt{\mu}\e})=0,$$ since $\mu\neq 0,$ we have $\mu=\alpha^{2},$ i.e., $\lambda=-\alpha^2.$ It follows from $\psi'(0)+\alpha\psi(0)=0$ that the eigenfunction associated with the negative eigenvalue $\lambda=-\alpha^2$ is $$\psi(x)=ce^{-\alpha x},\quad\mbox{with}\quad  c^{-1}=\|e^{-\alpha x}\|_{\LL^{2}(I)}.$$
\par If $\lambda=0$ then the general solution is $\psi(x)=Ax+B$ and from the boundary conditions it follows that $A=B=0.$ Therefore, $\lambda=0$ is not an eigenvalue.

\par In short, $\lambda_{0}\ii=-\alpha^2$ is the first eigenvalues (negative) of the Robin Laplacian on $I$ ( recall the self-adjointness of $-\Delta_{\alpha}^{I}$) associated with the normalized eigenfunction $\phi_{0}(y)=c e^{-\alpha y},\,c>0,$ i.e.,
$$-\phi_{0}^{''}=\lambda_{0}\phi_{0},\quad 0<\phi_{0}\in H^{2}(I)\quad\mbox{and}\quad \dis\int_{0}^{\e}|\phi_{0}|^2\,\ud y=1,$$
moreover, it satisfies the Robin boundary conditions \eqref{18J}. 

\section{ Infinite and straight planar strips}\label{LRF}

\par  The purpose of this section is to found that the classic Laplacian $-\Delta_{\alpha}^{\Omega_{\e}}$  in $\LL^{2}(\Omega_{\e}),$ with a suitable domain, is self-adjoint. For this purpose a convenient sesquilinear form $b_{\alpha}^{\Omega_\epsilon}$ has been introduced, whose definition will be made precise below.
\par Under certain conditions on~$\alpha$ at infinity, it is possible to prove the existence of isolated bound states, i.e., the existence of eigenvalues (of finite multiplicity) below the essential spectrum $\sigma_{\mathrm{ess}}(-\Delta_{\alpha}^{\Omega_{\e}})$ of Laplacian. For this purpose, we follow some ideas in~\cite{BK,GJ,MJ}. 

\par The closed sesquilinear form of interest is $b_{\alpha}^{\Omega_\epsilon}\geq- \|\alpha\|_{\infty}^{2}$ in $\LL^{2}(\Omega_{\e}),$ 
with domain the Hilbert space $H^{1}(\Omega_{\epsilon}),$
\begin{align*}
b_{\alpha}^{\Omega_\epsilon }(\phi,\psi)&=\int_{\Omega_\e}\overline{\nabla\phi}(x,y)\,\nabla\psi(x,y)\, \ud x\ud y\\
&+\int_{\R}^{}\alpha(x)\Big(\overline{\mathrm{tr}(\phi)}(x,\e)\mathrm{tr}(\psi)(x,\e)-\overline{\mathrm{tr}(\phi)}(x,0)\mathrm{tr}(\psi)(x,0)\Big)\,\ud x,
\end{align*}
where $\mathrm{tr}(\phi)$ denotes the trace in $\LL^{2}(\pt \Omega_{\e})$ of $\phi\in H^{1}(\Omega_\e).$
Consider the negative Laplacian $$-\Delta_{\alpha}^{\Omega_{\epsilon}}:\mathrm{dom}\,(-\Delta_{\alpha}^{\Omega_{\epsilon}}) \rightarrow \LL^{2}(\Omega_{\epsilon}),$$
with $\mathrm{dom}\,(-\Delta_{\alpha}^{\Omega_{\epsilon}})=\{\psi \in H^{2}(\Omega_\epsilon)\,;\psi\,\,\mbox{satisfies}\,\,(\ref{R})\},$
\begin{equation}\label{R}
\left \{
\begin{array}{cc}
-\dis\frac{\pt \psi}{\pt y}(x,0)- \alpha(x) \psi(x,0)=0\\ 
\,\,\,\,\dis\frac{\pt \psi}{\pt y}(x,\epsilon)+  \alpha(x) \psi(x,\epsilon)=0
\end{array}
\right.;
\end{equation}
in this section we require $\alpha(x)\in  W^{1,\,\infty}(\R).$ 
\par A lower bound for $b_{\alpha}^{\Omega_\epsilon }$ is  initially obtained for $\phi|_{\Omega_{\e}}$ with $\phi\in \CC_{0}^{\infty}(\R^{2})$ and, by density, for each $\phi\in H^{1}(\Omega_\e).$ Note that $\mathrm{tr}(\phi|_{\Omega_{\e}})=\phi|_{\Omega_{\e}}$ in $\LL^{2}(\pt\Omega_\e),$ consequently
\begin{eqnarray*}
b_{\alpha}^{\Omega_\epsilon }(\phi) \geq  \int_{\R}\left[\int_{0}^{\e} \left|\frac{\pt \phi}{\pt y}\right|^2\,\ud y +
\alpha(x)\big(|\phi(x,\e)|^2-|\phi(x,0)|^2\big)\right]\,\ud x,
\end{eqnarray*} 
since $\phi(x,\cdot)\in H^{1}(0,\e)$ for a.e.\ $x\in\R.$ By an argument in Theorem~\ref{laplaciano1}, we obtain
$$b_{\alpha}^{\Omega_\epsilon }(\phi)\geq-\|\alpha\|_{\infty}^{2}\int_{\Omega_{\e}}|\phi|^2\,\ud x\ud y,$$   
for each $\phi|_{\Omega_\e}$ with $\phi\in\CC_{0}^{\infty}(\R^{2}).$
Let $\phi\in H^{1}(\Omega_\e)$ and $\{\phi_{m}\}_{m=1}^{\infty}$ a sequence in $\CC_{0}^{\infty}(\R^{2})$ such that $\phi_{m}|_{\Omega_\e} \rightarrow \phi$ in $H^{1}(\Omega_\e).$ Then, for each $\phi\in H^{1}(\Omega_\e)$ one finds that 
\begin{equation}\label{linf}
b_{\alpha}^{\Omega_\epsilon }(\phi) \geq -\beta\|\phi\|^2,\quad \mbox{with}\quad \beta=\|\alpha\|_{\infty}^{2}.
\end{equation}
By standard arguments we can verify that $b_{\alpha}^{\Omega_\e }$ is closed. 

\par In the next theorem we show the self-adjointness of  $-\Delta_{\alpha}^{\Omega_\e},$ i.e., that $-\Delta_{\alpha}^{\Omega_\e}$ is the Rodin Laplacian, what we mean as the operator associated with the form $b_{\alpha}^{\Omega_\epsilon }.$ 
\begin{teo}\label{alpha}
Suppose that $\alpha \in W^{1,\,\,\infty}(\R).$ Then, the negative Laplacian  $-\Delta_{\alpha}^{\Omega_\epsilon}$ is the (unique) self-adjoint operator associated with the sesquilinear form  $b_{\alpha}^{\Omega_\epsilon},$ i.e., 
\[
b_{\alpha}^{\Omega_\epsilon}(\phi,\psi)=(\phi,-\Delta_{\alpha}^{\Omega_\e}\psi)_{\LL^{2}}
\]
for $\phi\in \mathrm{dom}\,\,b_{\alpha}^{\Omega_\epsilon },\,\psi\in\mathrm{dom}(-\Delta_{\alpha}^{\Omega_\epsilon }).$ 
\end{teo}

The proof is presented through Lemmas~\ref{include} and~\ref{extension}. The first lemma gives some information on the domain of  $T_{b_{\alpha}^{\Omega_\e}},$ associated with $b_{\alpha}^{\Omega_\e}.$ It shows that $\mathrm{dom}\,T_{b_{\alpha}^{\Omega_\e}}\subset \mathrm{dom}(-\Delta_{\alpha}^{\Omega_\e}).$
The second one concludes that $T_{b_{\alpha}^{\Omega_\e}}$ is an extension of $-\Delta_{\alpha}^{\Omega_\e}.$ Therefore, we obtain the equality $T_{b_{\alpha}^{\Omega_\e}}=-\Delta_{\alpha}^{\Omega_\e}.$ 
\begin{lemma}\label{include}
Suppose $\alpha\in W^{1,\,\,\infty}(\R).$ For each $F\in \LL^{2}(\Omega_\e),$ every solution $\psi\in H^{1}(\Omega_\e)$ of problem
\begin{equation}\label{2.1}
b_{\alpha}^{\Omega_\e}(\phi,\psi)=(\phi,F)_{\LL^{2}(\Omega_\e)},\,\,\forall\,\,\phi\in \mathrm{dom}\,b_{\alpha}^{\Omega_\e}=H^{1}(\Omega_\e),
\end{equation}
belongs to $\mathrm{dom}(-\Delta_{\alpha}^{\Omega_\e}).$ Consequently,  
$\mathrm{dom}\,T_{b_{\alpha}^{\Omega_\e}}\subset \mathrm{dom}(-\Delta_{\alpha}^{\Omega_\e}).$
\end{lemma}
\begin{proof}\,\,
For $\psi\in H^{1}(\Omega_\e),$ let us introduce the quotient of Newton
$$\psi_{\delta}(x,y)\ii=\frac{\psi(x+\delta,y)-\psi(x,y)}{\delta},\,\,0\neq\delta\in\R.$$  Since
\begin{eqnarray*}
|\psi(x+\delta,y)-\psi(x,y)|=
\left|\int_{0}^{1}\frac{\pt \psi}{\pt x}(x+\delta t,y)\delta\,\ud t\right|
\leq |\delta|\int_{0}^{1}\left|\frac{\pt \psi}{\pt x}(x+\delta t,y)\right|\,\ud t,
\end{eqnarray*}
we have
\begin{eqnarray*}
\int_{\Omega_\e}|\psi_{\delta}|^2\,\ud x\ud y\leq
\int_{0}^{1}\left[\int_{\Omega_\e}\left|\frac{\pt \psi}{\pt x}(x+\delta t,y)\right|^2\ud x\ud y\right]\,\ud t
=\int_{\Omega_\e}\left|\frac{\pt \psi}{\pt x}(x,y)\right|^2\ud x\ud y.
\end{eqnarray*}
Therefore, 
\begin{equation}\label{2.2}
\int_{\Omega_\e}|\psi_{\delta}|^2\,\ud x\ud y\leq \|\psi\|_{1,2}^{2},\quad\forall\;0\neq\delta\in\R .
\end{equation}
If $\psi\in H^{1}(\Omega_\e)$ is a solution to~(\ref{2.1}), then $\psi_{\delta}$ is a solution to the problem
$$b_{\alpha}^{\Omega_\e}(\phi,\psi_{\delta})=(\phi,F_{\delta})_{\LL^{2}(\Omega_\e)}-\int_{\R}\alpha_{\delta}(x)\Big(\overline{\phi(x,\e)}\psi(x+\delta,\e)-\overline{\phi(x,0)}\psi(x+\delta,0)\Big)\,\ud x,
$$ 
for each $\phi\in H^{1}(\Omega_\e).$ By chosing $\phi=\psi_{\delta}$ and noting that $(\phi,F_{\delta})_{\LL^{2}(\Omega_\e)}=-(\phi_{-\delta},F)_{\LL^{2}(\Omega_\e)},$  we obtain that
\begin{equation}\label{2.3}
b_{\alpha}^{\Omega_\e}(\psi_{\delta})=-((\psi_{\delta})_{-\delta},F)_{\LL^{2}(\Omega_\e)}-\int_{\R}\alpha_{\delta}(x)
\Big(\overline{\psi_{\delta}(x,\e)}\psi(x+\delta,\e)-\overline{\psi_{\delta}(x,0)}\psi(x+\delta,0)\Big)\,\ud x.
\end{equation}
For simplicity we write $b_{\alpha}^{\Omega}(\psi_{\delta})=b_{1}^{\Omega}(\psi_{\delta})+b_{2}^{\Omega}(\psi_{\delta}),$
with $$b_{1}(\psi_{\delta})=\int_{\Omega}|\nabla \psi_{\delta}|^{2}\,\ud x\ud y\quad\mbox{and}\quad
b_{2}(\psi_{\delta})=\int_{\R}\alpha(x)\Big( |\psi_{\delta}(x,\e)|^2 - |\psi_{\delta}(x,0)|^2 \Big)\,\ud x\ud y.$$
By Schwarz inequality, Cauchy inequality, estimate (\ref{2.2}), boundedness of~$\alpha$ and $\alpha_{\delta},$ and compact embedding of $H^{1}(\Omega_\e)$ in $\LL^2(\pt \Omega_\e),$ we can produce the following estimates, for $t>0,$
$$
|((\psi_{\delta})_{-\delta},F)_{\LL^{2}(\Omega_\e)}|\leq 2\|F\|_{\LL^{2}(\Omega_\e)}\|((\psi_{\delta})_{-\delta}\|_{\LL^{2}(\Omega_\e)}
\leq t^{-1}\|F\|_{\LL^{2}(\Omega_\e)}^{2}+t\|\psi_{\delta}\|_{1,2}^2,
$$
\begin{eqnarray*}
\left|\int_{\R}\alpha_{\delta}(x)
\Big(\overline{\psi_{\delta}(x,\e)}\psi(x+\delta,\e)-\overline{\psi_{\delta}(x,0)}\psi(x+\delta,0)\Big)\,\ud x\right|\\
\leq \|\alpha\|_{\infty}\|\psi_{\delta}\|_{\LL^{2}(\pt \Omega_\e)}\|\psi\|_{\LL^{2}(\pt \Omega_\e)} \leq C\|\psi_{\delta}\|_{1,2}\|\psi\|_{1,2},
\end{eqnarray*}
with $C>0$ independent of $\delta,$ 
\begin{align*}
\left|b_{2}^{\Omega}(\psi_{\delta})\right| &\leq  \left |\int_{\Omega}\alpha(x)\frac{\pt}{\pt y}|\psi_{\delta}|^{2}\,\ud x\ud y\right|
\leq 2\|\alpha\|_{\infty}\| \psi_{\delta}\|_{\LL^2(\Omega_\e)}\|\pt_{2}\psi_{\delta}\|_{\LL^{2}(\Omega_{\e})}
\leq t^{-1}\|\alpha\|_{\infty}^{2} \|\psi\|_{1,2}^{2}+tb_{1}^{\Omega}(\psi_{\delta}).
\end{align*}
On the one hand, provided that $b_{\alpha}^{\Omega}(\psi_{\delta})=b_{1}^{\Omega}(\psi_{\delta})+b_{2}^{\Omega}(\psi_{\delta}),$ one has
$$b_{\alpha}^{\Omega}(\psi_{\delta})\geq (1-t)b_{1}^{\Omega}(\psi_{\delta})-t^{-1}\|\alpha\|_{\infty}^{2} \|\psi\|_{1,2}^{2} .$$ 
On the other hand, the identity (\ref{2.3}) produces
$$\left|b_{\alpha}^{\Omega}(\psi_{\delta})\right|\leq C\|\psi_{\delta}\|_{1,2}\|\psi\|_{1,2}+\Big(t^{-1}\|F\|_{\LL^{2}(\Omega_\e)}^{2}+t\|\psi_{\delta}\|_{1,2}^2\Big).$$
So,
$$(1-t)b_{1}^{\Omega}(\psi_{\delta})-t^{-1}\|\alpha\|_{\infty}^{2} \|\psi\|_{1,2}^{2}\leq C\|\psi_{\delta}\|_{1,2}\|\psi\|_{1,2}+\Big(t^{-1}\|F\|_{\LL^{2}(\Omega_\e)}^{2}+t\|\psi_{\delta}\|_{1,2}^2\Big).$$
Now, suppose that $0<t<1$ and add $(1-t)\|\psi_{\delta}\|_{2}^{2}$ to both sides of the above inequality, to obtain
\begin{align*}
(1-t)\|\psi_{\delta}\|_{1,2}^{2}&\leq\CC\|\psi_{\delta}\|_{1,2}\|\psi\|_{1,2}+t^{-1}\|F\|_{\LL^{2}(\Omega_\e)}^{2}+t\|\psi_{\delta}\|_{1,2}^2+
t^{-1}\|\alpha\|_{\infty}^{2} \|\psi\|_{1,2}^{2}\\
&+(1-t)\|\psi_{\delta}\|_{2}^{2}.
\end{align*}
Therefore,
\begin{eqnarray*}
0\leq (2t-1)\|\psi_{\delta}\|_{1,2}^{2}+C\|\psi_{\delta}\|_{1,2}\|\psi\|_{1,2}+\Big(t^{-1}\|F\|_{\LL^{2}(\Omega_\e)}^{2}+
t^{-1}\|\alpha\|_{\infty}^{2} \|\psi\|_{1,2}^{2}+(1-t)\|\psi\|_{1,2}^{2}\Big).
\end{eqnarray*}
Thus, we assume that $0<t<1/2$, so that the dominant term of the quadratic function is negative,  consequently we have
$\|\psi_{\delta}\|_{1,2} ^{2}\leq \tilde{C},$ with $\tilde{C}$ independent of $\delta.$ But, this estimate implies
$$\sup_{\delta}\| \psi_{-\delta} \|_{1,2}< \infty,$$ and since $H^{1}(\Omega_\e)$ is reflexive, every bounded sequence has a weakly convergent subsequence, then there is $v\in H^{1}(\Omega_\e)$ and a subsequence $\delta_{k}\rightarrow 0$ such that $\psi_{-\delta_{k}}\stackrel{w}{\rightarrow} v$ in $H^{1}(\Omega_\e).$ Hence,
\begin{eqnarray*}
\int_{\Omega_\e}\psi\pt_{x}\phi\,\ud x\,\ud y=\int_{\Omega_\e}\psi \lim_{\delta_{k}\rightarrow 0}\phi_{\delta_{k}}\,\ud x\ud y
=\lim_{\delta_{k}\rightarrow 0}\int_{\Omega_\e}\psi\phi_{\delta_{k}}\,\ud x\,\ud y\\
=-\lim_{\delta_k\rightarrow 0}\int_{\Omega_\e}\psi_{-\delta_{k}}\phi\,\ud x\ud y
=-\int_{\Omega_\e}v\phi\,\ud x\ud y.
\end{eqnarray*}
Therefore, $\pt_{x}\psi=v$ 
in the weak sense, and so $\pt_{x}\psi\in H^{1}(\Omega_\e).$ Consequently, 
$\pt_{xx}\psi\in \LL^{2}(\Omega_\e)$ and $\pt_{yx}\psi\in \LL^{2}(\Omega_\e).$  It follows from the standard elliptic regularity theorems (see~\cite{E} Theorem~1, page~309) that
$\psi\in W_{loc}^{2,2}(\Omega_\e),$ so $-\Delta \psi=F$ a.e.\ in $\Omega_{\e}.$ Hence, 
$\pt_{yy}\psi=-(F+\pt_{xx}\psi)\in \LL^{2}(\Omega_\e),$ and therefore $\psi\in W^{2,2}(\Omega_\e).$
\par Finally, it remains to verify that $\psi$ satisfies the boundary conditions. After integration by parts 
\begin{align*}
(\phi,F)_{\LL^{2}(\Omega_\e)}=b_{\alpha}^{\Omega_\e}(\psi,\phi)=(\phi,-\Delta\psi)_{\LL^{2}(\Omega_\e)}+\int_{\R}\overline{\phi(x,0)}[-\pt_{y}\psi(x,0)-\alpha(x)\psi(x,0)]\,\ud x\\
+\int_{\R}\overline{\phi(x,\e)}[\pt_{y}\psi(x,\e)+\alpha(x)\psi(x,\e)]\,\ud x
\end{align*}
for each $\phi\in H^{1}(\Omega_\e).$ This implies the boundary conditions, because
$-\Delta \psi=F$ a.e.\ in $\Omega_\e$ and $\phi$ is  arbitrary. 
\end{proof}

\begin{lemma}\label{extension}
Suppose that $\alpha\in W^{1,\,\,\infty}(\R)$. Then $T_{\alpha}=-\Delta_{\alpha}^{\Omega_\e}.$
\end{lemma}
\begin{proof}\,\, Let $\psi\in\mathrm{dom}\,(-\Delta_{\alpha}^{\Omega_\e}),$ then $\psi\in W^{2,\,2}(\Omega_\e)$ and it satisfies the 
boundary conditions~(\ref{R}). By integration by parts and (\ref{R}) we obtain, for each  $\phi\in \mathrm{dom}\,b_{\alpha}^{\Omega_\e},$ the identity
\begin{eqnarray*}
b_{\alpha}^{\Omega_\e}(\phi,\psi)=\int_{\R}\overline{\phi(x,\e)} \pt_{y}\psi(x,\e)\,\ud x-\int_{\R}\overline{\phi(x,0)}\pt_{y}\psi(x,0)\,\ud x
-\int_{\Omega}\overline{\phi(x,y)} \Delta\psi(x,y)\,\ud x\,\ud y\\
+\int_{\R}\alpha(x)\overline{\phi(x,\e)}\psi(x,\e)\,\ud x
-\int_{\R}\alpha(x)\overline{\phi(x,0)}\psi(x,0)\,\ud x
=(\phi,-\Delta \psi)_{\LL^{2}(\Omega_{\e})}.
\end{eqnarray*}
Thus, $\psi \in \mathrm{dom}\,T_{\alpha},$ and it follows that $T_{\alpha}$ is an extension of $-\Delta_{\alpha}^{\Omega_\e}.$ It follows, by Lemma~\ref{include}, the desired equality.
\end{proof}

\section{ The spectrum of the Robin Laplacian in $\Omega_{\e}$ }\label{discrete}
 
Here we investigate the spectrum of the operator $-\Delta_{\alpha}^{\Omega_\e}$ when the Robin parameter (function)~$\alpha$ in $W^{1,\,\infty}(\R)$ satisfies the condition 
\begin{equation}\label{Con}
\dis\lim_{|x|\to+\infty}(\alpha(x)-\alpha_{0})=0,
\end{equation} 
i.e, given $\delta>0$ there exists $a>0$ such that $|\alpha(x)-\alpha_{0}|<\delta$ whenever $|x|>a.$ Denote $\beta\ii=(\alpha-\alpha_0)$ and define the functions
\begin{equation}\label{uniforme}
\beta_{m}=\left \{
\begin{array}{cc}
\alpha-\alpha_0,\quad \mathrm{if} \quad |x|<m\\ 
0\quad\quad\,\,\,,\quad \mathrm{if}\quad |x|\geq m\end{array}
\right..
\end{equation} 
This sequence of bounded  functions with compact support converges to $\beta$ in $\LL^{\infty}(\R)$.
\par In case that \eqref{Con} holds, we prove that the essential part $\sigma_{\mathrm{ess}}(-\Delta_{\alpha}^{\Omega_\e})$ of the spectrum $-\Delta_{\alpha}^{\Omega_\e}$ is the interval $[-\alpha_{0}^{2},\infty).$ This statement is the content of Theorem~\ref{ess}, whose proof is performed in two steps, that is, Propositions~\ref{CPEss} and and~\ref{Compact}, whose proofs were inspired in~\cite{BK}. The proof of Proposition~\ref{CPEss} makes use of the so-called called Weyl criterion for the essential spectrum  (see~\cite{CEsar}, Theorem 11.2.7), which we recall.

\begin{lemma}(Weyl criterion)\label{5.1}
Let $T$ be a self-adjoint operator in a complex Hilbert space $\mathcal{H}.$ Then, $\lambda\in \sigma_{\mathrm{ess}}(T)$ iff there exists a sequence $\{\psi_n\}_{n=1}^{\infty}\subset \mathrm{dom}\,T$ such that
\begin{itemize}
\item [1)]$ \|\psi_n\|=1,\,\,\, \forall\, n\in \N;$
\item [2)]$\psi_ {n}\stackrel{w}{\longrightarrow}0,$ as $n\to\infty$ in $\mathcal{H};$
\item [3)]$(T-\lambda)\psi_{n}\to0$, as $n\to\infty.$
\end{itemize}
Such a sequence is called a singular Weyl sequence for $T$ at $\lambda.$
\end{lemma}

\begin{lemma}\label{ess}
For each $\alpha_0\in\mathbb R$,  $[-\alpha_{0}^{2},+\infty)\subset \sigma_{\mathrm{ess}}(-\Delta_{\alpha_{0}}^{\Omega_\e}).$
\end{lemma}
\begin{proof}\,\,
Let $\lambda \in [\mu_{0},\infty)$ with $\mu_{0}=-\alpha_{0}^2.$ So,  one can write $\lambda=\mu_{0}+t,$ with $t\in[0,+\infty).$ 
We denote by $-\Delta^{\R}$ the  Laplacian operator in $\LL^{2}(\R).$ It is well known that $ \sigma_{\mathrm{ess}}(-\Delta^{\R})=[0,+\infty).$
Hence, there is a singular Weyl sequence $\{\phi_{n}\}_{n=1}^{+\infty}$ for $-\Delta^{\R}$ at $t.$ Define the sequence $\{\psi_{n}\}_{n=1}^{+\infty}$ as $\psi_{n}(x,y)=\phi_{n}(x)\phi_{0}(y)$ with $\phi_{0}=c(\alpha_0)e^{-\alpha_0 y}$ the eigenfunction (normalized) of $-\Delta_{\alpha_0}^{I},$ associated with the first eigenvalue $\mu_{0}=-\alpha_0^{2}.$ Note that $\{\psi_n\}_{n=1}^{\infty}\subset \mathrm{dom}\,(-\Delta_{\alpha_{0}}).$ It is easy to check that $\|\psi_{n}\|_{\LL^{2}(\Omega_\e)}=1,$ for each $n\in\N,$ and $\psi_ {n}\stackrel{w}{\longrightarrow}0,$ in $\LL^{2}(\Omega_\e),$ and also that $(-\Delta_{\alpha}^{\Omega_\e}-\lambda)\psi_{n}\to0,$ in $\LL^{2}(\Omega),$ since
$$(-\Delta_{\alpha_0}^{\Omega_\e}-\lambda)\psi_{n}=[(-\Delta^{\R}-t)\phi_{n}]\phi_{0}+[(-\Delta_{\alpha_0}^{I}-\mu_{0})\phi_{0}]\phi_{n}\;\longrightarrow \;0.$$
Hence $\{\psi_n\}_{n=1}^{\infty}$ is a singular Weyl sequence for $-\Delta_{\alpha_{0}}^{\Omega_\e}$ at $\lambda$ and, by Lemma~\ref{5.1}, 
$\lambda\in \sigma_{\mathrm{ess}}(-\Delta_{\alpha_{0}}^{\Omega_\e}).$
\end{proof}

\begin{proposition}\label{CPEss}
For each $\alpha_0\in\mathbb R$,  $\sigma_{\mathrm{ess}}(-\Delta_{\alpha_{0}}^{\Omega_\e})=[-\alpha_{0}^{2},\infty).$
\end{proposition}
\begin{proof}\,\, 
Indeed, by Lemma~\ref{ess} we have that $[-\alpha_{0}^{2},\infty)\subset \sigma_{\mathrm{ess}}(-\Delta_{\alpha_{0}}^{\Omega_\e}).$
On the other hand, we have the lower bound $b_{\alpha_{0}}^{\Omega_\e}\geq -\alpha_{0}^2,$ consequently $\sigma_{\mathrm{ess}}(-\Delta_{\alpha_{0}}^{\Omega_\e})\subset [-\alpha_{0}^{2},\infty).$ Therefore, the following holds
$$\sigma_{\mathrm{ess}}(-\Delta_{\alpha_0}^{\Omega_\e})=[-\alpha_{0}^{2},\infty).$$
\end{proof}

\begin{lemma}\label{PT-Symmetric Waveguides}
Let $\alpha_{0}\in\R$ and $\varphi\in\LL^2(\pt \Omega_\e)$. Then, there exists a positive constant~$C,$ depending on~$\e$ and $|\alpha_0|,$ such that any solution $\psi\in W^{2,2}(\Omega_\e)$ of the
boundary value problem
\begin{equation}\label{Aux}
\left \{
\begin{array}{cc}
(-\Delta-\lambda)\psi=0\quad\mbox{in}\quad\Omega_\e\\ 
-\dis\frac{\pt \psi}{\pt y}(x,0)- \alpha_{0} \psi(x,0)=\varphi(x,0)\\
\,\,\,\,\dis\frac{\pt \psi}{\pt y}(x,\epsilon)+  \alpha_{0} \psi(x,\epsilon)=\varphi(x,\e)
\end{array}
\right.,
\end{equation}
with any $\lambda <0,$ satisfies the estimate 
\begin{equation}\label{limited}
\|\psi\|_{1,2}\leq C \|\varphi\|_{\LL^{2}(\pt\Omega_\e)}.
\end{equation}
\end{lemma}
\begin{proof}\,\,
Multiplying the first equation of \eqref{Aux} by $\overline{\psi}$ and integrating by parts, one can produce the identity
$$\int_{\Omega_\e}|\nabla\psi|^2\,\ud x\ud y+\alpha_{0}\int_{\pt \Omega_\e}|\psi|^{2}\nu_{2}\,\ud \sigma-\lambda\int_{\Omega_\e}|\psi|^{2}\,\ud x\ud y
=\int_{\pt \Omega_\e}\varphi\overline{\psi}\nu_{2}\,\ud \sigma,$$ where $\nu_2$ denotes the second component of the outward unit normal vector
to $\pt \Omega_\e.$ Using the Schwarz and Cauchy inequalities, recalling that $|\nu_2| = 1,$  and the embedding of $H^{1}(\Omega_\e)$ in $\LL^2(\pt \Omega_\e),$  we have, for $t\in(0,1),$
\begin{align*}
\left|\int_{\pt \Omega}\alpha_0|\psi|^{2}\nu_{2}\,\ud \sigma\right|&=\left |\int_{\Omega}\alpha_0\frac{\pt}{\pt y}|\psi|^{2}\,\ud x\ud y\right|
\leq 2|\alpha_0|\| \psi\|_{\LL^2(\Omega_\e)}\|\pt_{2}\psi\|_{\LL^{2}(\Omega_{\e})}\\
{}&\leq t^{-1}|\alpha_0|^{2}\|\psi\|_{\LL^2(\Omega_\e)}^{2}+t\|\nabla\psi\|_{\LL^{2}(\Omega_{\e})}^{2}
\leq t^{-1}|\alpha_0|^{2}\|\psi\|_{1,2}^{2}+t\|\nabla\psi\|_{\LL^{2}(\Omega_{\e})}^{2}\,,\\
\left|\int_{\pt \Omega}\varphi\overline{\psi}\nu_2\,\ud \sigma\right|&\leq 2\|\psi\|_{\LL^{2}(\pt\Omega_\e)}\|\varphi\|_{\LL^{2}(\pt\Omega_\e)}
\leq t^{-1}\|\varphi\|_{\LL^{2}(\pt\Omega_\e)}^{2}+t\|\psi\|_{\LL^{2}(\pt\Omega_{\e})}^{2}\\
{}&\leq t^{-1}\|\varphi\|_{\LL^{2}(\pt\Omega_\e)}^{2}+t\tilde{C}\|\psi\|_{1,2}^{2},
\end{align*}
where $\tilde{C}$ is the constant from the embedding of $H^{1}(\Omega_\e)$ in $\LL^{2}(\pt\Omega_\e).$ By the above estimates, we obtain
$$\big(1-t-\lambda-t|\alpha_0|^{2}-t\tilde{C}\big)\|\psi\|_{1,2}^{2}\leq t^{-1}\|\varphi\|_{\LL^{2}(\pt\Omega_\e)}^{2}.$$
The desired conclusion follows by choosing $t>0$ small enough so that the coefficient of $\|\psi\|_{1,2}^{2}$  becomes positive.
\end{proof}

\begin{proposition}\label{Compact}
Suppose that $\alpha\in W^{1,\,\infty}(\R).$ If $\dis\lim_{|x|\to+\infty}(\alpha(x)-\alpha_{0})=0$, then for each $\lambda\in \rho(-\Delta_{\alpha}^{\Omega_\e})\cap \rho(-\Delta_{\alpha_0}^{\Omega_\e})$ the operator $(-\Delta_{\alpha}^{\Omega_\e}-\lambda)^{-1}-(-\Delta_{\alpha_0}^{\Omega_\e}-\lambda)^{-1}$ is compact in $\LL^{2}(\Omega_\e).$ 
\end{proposition}
\begin{proof}\,\,
Due to the first resolvent identity, it is enough to prove the result for a negative~$\lambda$   in the intersection of the respective resolvent 
sets. Consider $\{\phi_{j}\}_{j=1}^{\infty} \subset \LL^{2}(\Omega_\e)$ bounded and let $\psi_j=(-\Delta_{\alpha}^{\Omega_\e}-\lambda)^{-1}\phi_j-(-\Delta_{\alpha_0}^{\Omega_\e}-\lambda)^{-1}\phi_j$; note that $\psi_j$ satisfies the first equation in \eqref{Aux}. Moreover, inserting $\psi_j$ into  the second or third equation  we obtain
$$\left(\frac{\pt}{\pt_{y}}+\alpha_{0}\right)\psi_j=\left(\frac{\pt}{\pt_{y}}+\alpha_{0}\right)\Big((-\Delta_{\alpha}^{\Omega_\e}-\lambda)^{-1}\phi_j-(-\Delta_{\alpha_0}^{\Omega_\e}-\lambda)^{-1}\phi_j\Big)=(\alpha_0-\alpha)\mathrm{tr}(-\Delta_{\alpha}^{\Omega_\e}-\lambda)^{-1}\phi_j,$$ so that we take now $\varphi=(\alpha_0-\alpha)\mathrm{tr}(-\Delta_{\alpha}^{\Omega_\e}-\lambda)^{-1}\phi_j,$ where  $\mathrm{tr}$ denotes the trace operator from $H^{1}(\Omega_\e)\supset \mathrm{dom}\,(-\Delta_{\alpha}^{\Omega_\e})$ to $\LL^{2}(\pt\Omega_\e).$
By Lemma~\ref{PT-Symmetric Waveguides}, we have
$$\|\psi_j-\psi_k\|_{1,2}\leq C\left\|\big((\alpha_{0}-\alpha)\mathrm{tr}(-\Delta_{\alpha}^{\Omega_\e}-\lambda)^{-1}\big)(\phi_{j}-\phi_{k})\right\|_{\LL^{2}(\pt\Omega_\e)}.$$
Under the assumption that $\beta \mathrm{tr}(-\Delta_{\alpha}^{\Omega_\e}-\lambda)^{-1}$ is a compact operator, it follows that the sequence $\{\psi_j\}_{j= 1}^{\infty}$ is precompact in the topology of $H^{1}(\Omega_\e),$ and with a help of the above inequality  one can establish  that  $(-\Delta_{\alpha}^{\Omega_\e}-\lambda)^{-1}-(-\Delta_{\alpha_0}^{\Omega_\e}-\lambda)^{-1}$ is a compact operator in $\LL^{2}(\Omega_\e).$ 
 
Let us verify the compactness  of the operator $\beta \mathrm{tr}(-\Delta_{\alpha}^{\Omega_\e}-\lambda)^{-1}.$ One can show that the sequence of operators $\beta_{m}\mathrm{tr}(-\Delta_{\alpha}^{\Omega_\e}-\lambda)^{-1}$ converges to $\beta \mathrm{tr}(-\Delta_{\alpha}^{\Omega_\e}-\lambda)^{-1}$ in norm, since $\|\beta_m-\beta\|_{\LL^{\infty}(\R)}\to0$ in $\LL^{\infty}(\R)$.
On the other hand, we shall prove that each operator $\beta_{m}\mathrm{tr}(-\Delta_{\alpha}^{\Omega_\e}-\lambda)^{-1}$ is compact.  Indeed, given $\{u_{n}\}_{n=1}^{\infty}$ bounded in $\LL^{2}(\Omega_\e)$ one has $v_n=(-\Delta_{\alpha}^{\Omega_\e}-\lambda)^{-1}u_n$ bounded in $H^1(\Omega_\e)$, then there exists a subsequence, which we still denote by $v_n$, and a function $v\in H^1(\Omega_\e)$ such that $v_n\to v$ weakly in $H^1(\Omega_\e).$ Since $H^{1}(\Omega_\e)$ is compactly embedded in $\LL^{2}(\Omega_m),$ where $\Omega_m=(-m,m)\times(0,\e) \subset\Omega_\e,$ due to the Rellich-Kondrachov theorem (see \cite{Ad}, Sec.\ VI.), then $v_n\to v$ in $\LL^{2}(\Omega_m).$ According to the definition of $\beta_m,$ we have
\begin{align*}
\|\beta_{m}\mathrm{tr}(v_{n})-\beta_{m}\mathrm{tr}(v_{l})\|_{\LL^{2}(\pt\Omega_\e)}&=\|\beta \big(\mathrm{tr}(v_{n})-\mathrm{tr}(v_{l})\big)\|_{\LL^2(\pt\Omega_m)}
\leq C_{m}\|\beta\|_{\infty}\|v_n-v_l\|_{\LL^{2}(\Omega_m)}. 
\end{align*}
It follows that $\beta_{m}\mathrm{tr}(v_{n})$ is a Cauchy sequence and thus $\lim_{n\to\infty} \beta_{m}\mathrm{tr}(v_{n})$ exists, for each positive integer~$m$.
Therefore, the operators $\beta_{m}\mathrm{tr}(-\Delta_{\alpha}^{\Omega_\e}-\lambda)^{-1}$ are compact.
\end{proof}

\begin{teo}\label{ess2}
Let $\alpha \in W^{1,\,\infty}(\R).$ If $\dis\lim_{|x|\to+\infty}(\alpha(x)-\alpha_{0})=0$,  then
$$\sigma_{\mathrm{ess}}(-\Delta_{\alpha}^{\Omega_{\e}})=[-\alpha_{0}^{2}, \infty).$$ 
\end{teo}
\begin{proof}\,\,
According to Proposition~\ref{Compact}, the operator $(-\Delta_{\alpha}^{\Omega_\e}-\lambda)^{-1}-(-\Delta_{\alpha_0}^{\Omega_\e}-\lambda)^{-1}$ is compact in $\LL^{2}(\Omega_\e);$ then, by Theorem~XIII.14 in~\cite{RS}, the essential spectrum of $-\Delta_{\alpha}^{\Omega_{\e}}$ and $-\Delta_{\alpha_0}^{\Omega_{\e}}$ are identical.
\end{proof}


\subsection{Existence of  discrete spectrum} 

Now,  based on~\cite{GJ,MJ,KK} and under appropriate conditions, we shall give a variational argument to conclude that $\sigma(-\Delta_{\alpha}^{\Omega_{\e}})\cap (-\infty, -\alpha_{0}^{2})\neq\emptyset$. This, together with  Theorem~\ref{ess}, implies that the spectrum below $-\alpha_{0}^{2}$  is nonempty and formed by isolated eigenvalues of finite multiplicity, i.e., $\sigma_{\mathrm{disc}}(-\Delta_{\alpha}^{\Omega_\e})\neq \emptyset.$ 

\

\begin{teo}\label{boundstate}
Suppose that $(\alpha(x)-\alpha_{0}) \in W^{1,\,\infty}(\R)$, with $\alpha_{0}>0$ $(\alpha_{0}<0)$. If, moreover, $(\alpha(x)-\alpha_{0})$ is integrable with
$$\int_{\R}(\alpha(x)-\alpha_0)\,\ud x>0\quad\left( \int_{\R}(\alpha(x)-\alpha_0)\,\ud x<0 \right)$$ and $\dis\lim_{|x|\to+\infty}(\alpha(x)-\alpha_{0})=0$, then 
$$\inf\sigma(-\Delta_{\alpha}^{\Omega_\e})<-\alpha_{0}^{2}.$$
\end{teo}

\begin{proof}\,\,
Following \cite{GJ}, we wish to obtain a trial function $\psi$ from the form domain of  $-\Delta_{\alpha}^{\Omega_\e}$ such that the quadratic form 
$Q_{\alpha}^{\Omega_\e}(\psi)<0,$ where 
$$Q_{\alpha}^{\Omega_\e}(\phi)=b_{\alpha}^{\Omega_\e}(\phi)+\alpha_{0}^{2}\|\phi\|_{2}^{2},\quad \mathrm{dom}\,Q_{\alpha}^{\Omega_\e}=\mathrm{dom}\,b_{\alpha}^{\Omega_\e}.$$ Let $\zeta$ be a cut-off function, that is,
we fix a function $\zeta\in \CC_{0}^{\infty}(\R),$ with $0\leq\zeta\leq 1,$ and $\zeta\equiv1$ on $(-1/4,1/4),$ $\zeta\equiv0$ on $\R\backslash(-1/2,1/2)$ and $\|\zeta\|_{2}=1.$ Given $\phi_0$ as defined in the proof of Lemma~\ref{ess}. Consider the sequence $\{u_{n}\}_{n=1}^{\infty}$ of functions  into $\mathrm{dom}\,b_{\alpha}^{\Omega_\e},$ defined by $u_{n}(x,y)=f_n(x)\phi_{0}(y)$ where $f_{n}(x)=\zeta(x/n).$  
By integration by parts and using the boundary conditions of $\phi_{0},$ we obtain
$$
Q_{\alpha}^{\Omega_\e}(u_n)=n^{-1}\|\zeta^{'}\|_{2}^{2}+\|f_n\|_{2}^{2}\int_{0}^{\e}(|\pt_{y}\phi_{0}|^2+\alpha_{0}^{2}|\phi_{0}|^2)\,\ud y
+\int_{\R}\alpha(x)|f_{n}|^{2}(|\phi_{0}(\e)|^2-|\phi_{0}(0)|^2)\,\ud x.
$$
Since
$$\int_{0}^{\e}\Big(|\pt_{y}\phi_{0}|^2+\alpha_{0}^{2}|\phi_{0}|^2\Big)\,\ud y=-\alpha_{0}(|\phi_{0}(\e)|^2-|\phi_{0}(0)|^2),$$
we have
$$
Q_{\alpha}^{\Omega_\e}(u_n)=n^{-1}\|\zeta^{'}\|_{2}^{2}+(|\phi_{0}(\e)|^2-|\phi_{0}(0)|^2)\int_{\R}\big(\alpha(x)-\alpha_{0}\big)|f_{n}|^{2}\,\ud x,
$$
since $|f_{n}(x)(\alpha(x)-\alpha_0)|\leq|\alpha(x)-\alpha_0|$ and $f_{n}(x)\rightarrow 1$ as $n\rightarrow\infty$, we can apply the Dominated Convergence Theorem, because $|\alpha(x)-\alpha_0|\in\LL^{1}(\R),$  to get
$$\lim_{ n\rightarrow\infty}Q_{\alpha}^{\Omega_\e}(u_n)=(|\phi_{0}(\e)|^{2}-|\phi_{0}(0)|^{2})\int_{\R}(\alpha(x)-\alpha_0)\,\ud x<0.$$ 
Then,  there exists some $u_{N}\in\mathrm{dom}\,b_{\alpha}^{\Omega_\e}$ such that $b_{\alpha}^{\Omega_\e}(u_{N})<-\alpha_{0}^{2}.$  
It follows that, by invoking  Rayleigh-Ritz Theorem, $\inf\sigma(-\Delta_{\alpha}^{\Omega_\e})<-\alpha_{0}^{2}.$ Consequently, $\sigma(-\Delta_{\alpha}^{\Omega_{\e}})\cap (-\infty, -\alpha_{0}^{2})\neq\emptyset.$
\end{proof}

\section*{Acknowledgement}
I thank C\'esar Rog\'egio de Oliveira for valuable discussions. With readiness  he spent time answering my questions during our long discussions which have helped to improve this work. This work was supported financially by CAPES (Brazil).



\begin{thebibliography}{12}
\bibitem{Ad} Adams R.A.: \textit{Sobolev spaces}, Academic Press, New York, 1975.




\bibitem{BK} Borisov D., Krej\v{c}i\v{r}\'ik D.: PT-\textit{symmetric waveguide}, Integral Equations and Operator Theory \textbf{62}, 489-515.



\bibitem{CDFK} \textrm{Chenaud, B., Duclos, P., Freitas, P., and Krej\v{c}i\v{r}ík, D.}: \textit{Geometrically induced discrete spectrum in
curved tubes}, Differential Geometry and its Applications \textbf{23} (2005) 95–105.



\bibitem{CEsar}  de Oliveira, C.R.:  \textit{Intermediate Spectral Theory and Quantum Dynamics}, BirkhŠuser, Basel (2009). 



\bibitem{CAL}  \textrm{de Oliveira, C.R., Rossini, A.F.}: \textit{Effective operators for nonhomogeneous Robin Laplacian in  thin two- and three-dimensional curved waveguides}. Submitted for publication.

\bibitem{CA}  \textrm{de Oliveira, C.R., Verri, A.A.}: \textit{On the spectrum and weakly effective operator for Dirichlet  Laplacian in thin deformed tubes}, J. Math. Anal. Appl. \textbf{381} (2011), 454--468.

\bibitem{DK} \textrm {Dittrich, J., and K\v{r}\'i\v{z}, J.}: \textit{Bound states in straight quantum waveguides with combined
boundary conditions,} Journal of Mathematical Physics \textbf{43}, 8 (2002), 3892–3915.

\bibitem{DE} \textrm {Duclos, P., Exner, P.}: \textit{Curvature-induced bound states in quantum waveguides in two and three dimensions}, Rev. Math. Phys. \textbf{7} (1995), 73--102.


\bibitem{E} Evans L.C.: \textit{Partial differential equations},  American Mathematical Society, Providence, 1998.

\bibitem{EMP} \textrm {Exner, P., Minakov, A., and Leonid Parnovski, L.}: \textit{Asymptotic eigenvalue estimates for a Robin
problem with a large parameter}, Portugal. Math. (N.S.) Portugaliae Mathematica
Vol. 71, Fasc. 2, 2014, 141–156 6 European Mathematical Society
DOI 10.4171/PM/1945.


\bibitem{ES} \textrm{Exner, P., and  \v{S}eba, P.}: \textit{Bound states in curved waveguides}, J. Math. Phys. \textbf{30} (1989), 25742580.



\bibitem{FK}  \textrm{Freitas, P., Krej\v{c}i\v{r}ík, D.}: \textit{Waveguides with combined Dirichlet and Robin boundary conditions}, Math. Phys. Anal. Geom. \textbf{9} (2006), 335--352.


\bibitem{FSR}  \textrm{Friedlander, F., Solomyak, M.}: \textit{On the spectrum of the Laplacian in a narrow  strip}, Israel J. Math. \textbf{170} (2009), 337-354.

\bibitem{FS}  \textrm{Friedlander, F., Solomyak, M.}: \textit{On the spectrum of the Laplacian in a narrow infinite strip}, Amer. Math. Soc. Transl. \textbf{225}
 (2008), 103-116.


\bibitem{GJ}  \textrm{Goldstone J., Jaffe R.L.}: \textit{Bound states in twisting tubes},
Phys. Rev. B, \textbf{45} (1992), 14100-14107.


\bibitem{MJ}  \textrm{J\'ilek, M.}: \textit {Straight quantum waveguide with Robin boundary conditions}, Symmetry, Integrability and Geometry: Methods and Applications. SIGMA \textbf{3} (2007), 108 (12 pages). 

\bibitem{Ka} Kato, T.: \textit{Perturbation theory for linear operators}, Springer-Verlag, Berlin, 1966. 




\bibitem{KBZ} \textrm{Krej\v{c}i\v{r}\'ik, D., Bíla, H.  and Znojil, M.}: \textit{Closed formula
the metric in the Hilbert space of a PT-symmetric model.}, J. Phys. A \textbf{39} (2006), 10143--10153.

\bibitem{KK}  \textrm{Krej\v{c}i\v{r}\'ik, D., K\v{r}\'i\v{z}, J.}: \textit{On the spectrum of curved planar waveguides},
Publ. RIMS Kyoto Univ. \textbf{41} (2005), 757-791.

\bibitem{DZ} \textrm {Krej\v{c}i\v{r}ík, D. and Zhiqin Lu}: \textit{Location of the essential spectrum in curved quantum layers},
J. Math. Phys. \textbf{55}, (2014) 083520.

\bibitem{PK} \textrm {Pankrashkin, K.}: \textit {On the asymptotics of the principal eigenvalue for a Robin problem wint  a large parameter in planar domains}, Nanosystems: PhyS, Chemistry, Math, 2013, 4 (4), P. 474–483.

\bibitem{RS} \textrm{Reed M., Simon B.}: \textit{Methods of modern mathematical physics}, V. Analysis of operators, Academic Press,
New York, 1978. 


 


\end{thebibliography}
\end{document}